\documentclass[12pt]{article}
\usepackage[utf8]{inputenc}
\usepackage[T1]{fontenc}
\usepackage{mathpazo}
\RequirePackage[english]{babel}
\RequirePackage[a4paper,margin=2.5cm]{geometry}
\RequirePackage{parskip}
\newcommand{\affiliation}{\footnote}
\makeatletter
\def\@fnsymbol#1{\ensuremath{\ifcase#1\or *\or \dagger\or \ddagger\or \mathsection\or \|\or **\or \dagger\dagger \or \ddagger\ddagger \else\@ctrerr\fi}}
\makeatother
\usepackage{xcolor}
\definecolor{cblue}{RGB}{0,70,140}
\definecolor{cgreen}{RGB}{100,140,0}
\definecolor{cred}{RGB}{190,10,50}
\usepackage[shortlabels]{enumitem}
\setlist{itemsep=0ex,topsep=0ex,parsep=0.4ex}
\usepackage{amsmath,amsfonts,amssymb,amsthm,mathtools,thmtools,thm-restate,centernot}
\usepackage[hyphens]{url} % Line breaking of urls
\usepackage[hidelinks,colorlinks,pagebackref]{hyperref} % Coloring of hyperlinks and backreferences in bibliography; must be loaded after url
\hypersetup{citecolor=cgreen,linkcolor=cblue,urlcolor=cblue}
\usepackage[capitalise,nameinlink,noabbrev,compress]{cleveref} % References in the document; must be loaded after hyperref and amsmath

\renewcommand*{\backref}[1]{}
\renewcommand*{\backrefalt}[4]{
	\ifcase #1 Not cited.%
	\or $\uparrow$#2%
	\else $\uparrow$#2%
	\fi%
}
\newcommand{\comment}[1]{}

\theoremstyle{plain}
\newtheorem{theorem}{Theorem}[section]
\newtheorem{lemma}[theorem]{Lemma}
\newcommand{\proofitemfont}{\normalfont\itshape}

\newtheoremstyle{proofclaim}
  {3pt}              
  {3pt}              
  {\normalfont}      
  {}                 
  {\proofitemfont}  
  {.}                
  {.5em}            
  {}                 

\theoremstyle{proofclaim}
\newtheorem{claim}{Claim}
\theoremstyle{plain}
\newtheorem{proposition}[theorem]{Proposition}

\theoremstyle{definition}

\makeatletter
\renewenvironment{proof}[1][\proofname]
{\par\pushQED{\qed}
	\normalfont\topsep6\p@\@plus6\p@\relax\trivlist
	   \item[\hskip\labelsep
      \ifnum\proofdepth>1 \proofitemfont\else\bfseries\fi
      #1\@addpunct{.}]
	\ignorespaces}
{\popQED\endtrivlist\@endpefalse}
\makeatother
\let\emptyset\varnothing
\newcommand{\eps}{\varepsilon}

\newcounter{case}[section]
\renewcommand{\thecase}{\arabic{case}}

\newcommand{\caseheading}{%
  \refstepcounter{case}%
  \par\medskip\noindent
  {\proofitemfont Case \thecase.}\hspace*{1em}
  \ignorespaces
}

\newcounter{subcase}[case]

\newcommand{\ex}{{\rm ex}}

\author{Maria Axenovich\affiliation{Institute of Algebra and Geometry, Karlsruhe Institute of Technology, Germany (\textsf{\href{mailto:maria.aksenovich@kit.edu}{maria.aksenovich@kit.edu}})} \and {Marko Pejic \affiliation{Institute of Algebra and Geometry, Karlsruhe Institute of Technology, Germany (\textsf{\href{mailto:marko.pejic@student.kit.edu}{marko.pejic@student.kit.edu}})}}}

\title{On high-girth layered graphs of large Tur\'an density in a hypercube}

\newcount\proofdepth
\newcounter{proofnumber}

\AddToHook{env/proof/begin}{%
  \ifnum\proofdepth=0
    \stepcounter{proofnumber}%
    \setcounter{claim}{0}%
    \setcounter{case}{0}%
    \setcounter{subcase}{0}%
  \fi
  \advance\proofdepth by 1\relax
}

\begin{document}
    \maketitle
      {
    \hypersetup{linkcolor=cred}
    \newcommand{\etalchar}[1]{$^{#1}$}

\begin{abstract}
For a graph $H$, let $\ex(Q_n, H)$ be the largest number of edges in a subgraph of the hypercube $Q_n$ of dimension $n$ that contains no subgraph isomorphic to $H$. 
We define $\pi_\square(H)=\lim_{n\to\infty}\ex(Q_n,H)/|E(Q_n)|$ as the Tur\'an density of $H$ in the hypercube.
Although determining $\pi_\square(H)$ for general $H$ remains widely open, Conlon found a large class of graphs with zero Tur\'an density in a hypercube.
In this note, we address the case when $\pi_{\square} (H)>0$. 
If a graph $H$ cannot be embedded in an edge layer of a hypercube, then $\pi_{\square} (H)\geq 1/2$, as can be seen by taking every other edge layer of $Q_n$.
Previously known examples of layered graphs of positive
Tur\'an density in a hypercube include graphs containing
cycles of length $6$ or $10$, as well as examples of girth $8$.
Also, to the best of our knowledge, all previously known examples of
graphs with Tur\'an densities arbitrarily close to one are not layered.
We show that for every $g \geq 3$, there are layered graphs of girth greater than $g$ whose Tur\'an densities in a hypercube are arbitrarily close to one. 
Moreover, these graphs have an induced embedding into an edge layer
of a hypercube.
\end{abstract}

\section{Introduction}
A {\it hypercube} $Q_n$ with a ground set $X$ of size $n$ is a graph on a vertex set $\{A: A\subseteq X\}$ and an edge set consisting of all pairs $\{A,B\}$, where $A\subseteq B$ and $|A|=|B|-1$.    If a graph is a subgraph of $Q_n$, for some $n$, it is called {\it cubical}.  The $i$th {\it vertex layer} of $Q_n$ is the set of vertices of size $i$.  The $i$th {\it edge layer}  of $Q_n$ is the graph  induced by the union of the $i$th and $(i-1)$st vertex layers.  A graph is {\it layered} if it is a subgraph of some edge layer of a hypercube. 
For a graph $H$, let the {\it extremal number}  of $H$ in $Q_n$, denoted 	$\ex(Q_n, H)$,  be the largest number of edges in a subgraph $G$ of $Q_n$ such that there is no subgraph of $G$ isomorphic to $H$. Let $$\pi_\square(H)=\lim_{n\rightarrow \infty}  \ex(Q_n, H)/|E(Q_n)|$$
be the Tur\'an density of $H$ in a hypercube, also called the \textit{hypercube Tur\'an density of $H$}. 
A graph $H$ is said to have {\it zero Tur\'an density in a hypercube} if $\pi_\square(H)=0$. Otherwise, $H$ has a {\it positive Tur\'an density in a hypercube}. Here, when clear from context, we shall simply write Tur\'an density instead of Tur\'an density in the hypercube. By a standard double counting argument, the sequence
$\ex(Q_n,H)/|E(Q_n)|$ is nonincreasing, so the above density
notion is well defined.
The behaviour of the function $\ex(Q_n, H)$ is not well understood in general and there is no known  classification of  graphs with positive  Tur\'an density.

\paragraph{Related work.}
Early results on graphs of zero hypercube Tur\'an density were
obtained by Chung~\cite{Ch}.
Conlon~\cite{C} proved that a large class of graphs admitting
a partite representation has this property. This result connected graphs of zero hypercube Tur\'an density to uniform hypergraphs of zero Tur\'an density.  Later, the first author \cite{A-general} and Zhu \cite{Z} found other classes of graphs of zero hypercube Tur\'an density. Further classes of graphs, including certain subdivisions, with zero
Tur\'an density were identified by Axenovich, Martin and Winter,
see~\cite{AMW}. For more results on extremal numbers in a hypercube, see \cite{AKS,A-cycles,baber,balogh,BHN,FO,FO2,O}.

Examples of cubical graphs of positive Tur\'an density include non-layered graphs, in particular ones containing a $4$-cycle, as well as layered graphs containing a $6$-cycle or a $10$-cycle, see Chung \cite{Ch}, Conder \cite{conder}, and Grebennikov and Marciano \cite{GM}. 
Further layered examples with girth $8$ follow from the constructions of Alon, Radoi\v{c}i\'c, Sudakov, and Vondr\'ak~\cite[Section~4]{ARSV}. A graph can be recognised as layered by providing a  special edge colouring, see Axenovich, Martin, and Winter \cite{AMW}.
In addition, a recent, AI-assisted result, see Axenovich and Sagdeev \cite{AS}, shows that as $t$ grows, the largest size of a layered graph on $t$ vertices is half of a largest size of a cubical graph on $t$ vertices, that is $(1/4)t\log_2 t (1+o(1))$. Thus, any graph on $t$ vertices exceeding this size is not layered and thus has a positive hypercube Tur\'an density. 

Cubical graphs with large Tur\'an densities, that is, densities
arbitrarily close to one, include the subcubes $Q_d$.
Ellis, Ivan and Leader~\cite{EIL} proved that
$1-\pi_\square(Q_d)=O(c^d)$ as $d\rightarrow\infty$,
for some constant $0<c<1$. 
A refinement of their construction by Alon~\cite{A-erasure}
yields $1-\pi_\square(Q_d)\leq2^{-d/2+O(\log d)}$.
For earlier bounds and further results, see~\cite{AKS,O}.
However, all these graphs contain a $4$-cycle and hence are not layered.
Another related question is whether there are very sparse cubical graphs of positive hypercube Tur\'an density.
A positive answer to this question is given by Behague, Leader, Morrison, and Williams~\cite{BLMW}, who constructed cubical, non-layered graphs of arbitrarily high girth. 
The constructions just mentioned all concern non-layered cubical graphs, which motivates the following question.

\emph{Are there very sparse cubical layered graphs of positive hypercube Tur\'an density? How large can their Tur\'an densities be?}

In this note, we answer the first question in the affirmative by
constructing layered graphs of arbitrarily high girth
and show, in response to the second, that their Tur\'an densities
can even be arbitrarily close to one.
Moreover, these graphs have induced embeddings into edge layers of hypercubes and do not have the Ramsey property in the hypercube.

\paragraph{Our construction.}
For integers $d\geq419$ and $g\geq3$, let $H=H(d,g)$ be a
$d$-regular graph of girth at least $g$ with
$\lambda(H)\leq2\sqrt{d-1}+1$.
Here, $\lambda(H)=\max\{|\lambda_2|,|\lambda_{|V(H)|}|\}$,
where $d=\lambda_1\geq\ldots\geq\lambda_{|V(H)|}$ are the eigenvalues of the adjacency matrix of $H$, counted with multiplicity.
We shall show that such graphs exist using random regular graphs and standard spectral properties.
Also, the chromatic number of $H$ satisfies $\chi(H)\geq\lceil1+d/(2\sqrt{d-1}+1)\rceil\geq12$.

Our final graph is then the $1$-subdivision $T_1(H)$, obtained by
replacing every edge of $H$ with a path of length two.
This graph has an induced embedding into the second edge layer of a
hypercube and girth at least $2g$.
The crucial property of our graph is that, as we shall show, in every embedding of $T_1(H)$ into a hypercube,
the images of the vertices of $H$ have a common neighbour.
This property allows us to prove the following theorem, where
$q=q(H)$ is the largest prime power at most $(\chi(H)-4)/2$
and $\beta(q)=\prod_{j=1}^{\infty}(1-q^{-j}) > 0$.

\begin{theorem}\label{main}
There is a positive function $f$ with $f(d)\rightarrow\infty$
as $d\rightarrow\infty$ such that, for any integers
$d\geq419$ and $g\geq3$,
\[
    \pi_\square(T_1(H(d,g)))
    \geq \max\{1/2,\beta(q)^4\}
    > 1-1/f(d).
\]
Moreover, for every integer $n\geq1$, there is a colouring
of the edges of $Q_n$ in two colours containing no
monochromatic copy of $T_1(H(d,g))$.
\end{theorem}

\paragraph{Notation and definitions.}

All graphs are finite and simple. The ground set of $Q_n$ is $[n]=\{1,\ldots,n\}$, and $[0]=\emptyset$.
We write $A\triangle B$ for the symmetric difference of sets. For every $S\subseteq[n]$, the map $A\mapsto A\triangle S$ is an automorphism of $Q_n$.
For a graph $G$, we denote by $|V(G)|$ and $|E(G)|$ the numbers of its vertices
and edges, respectively. 
When convenient, we also use the shorthand notation $|G|=|V(G)|$ and $\|G\|=|E(G)|$.
For an edge $\{X,X\cup\{i\}\}$ of $Q_n$, its \textit{direction} is $i$.
For $S,T\subseteq V(G)$, let $e_G(S,T)$ denote the number of edges with one endpoint in $S$ and the other in $T$, counting an edge twice if both endpoints are in $S\cap T$.
For $V'\subseteq V(G)$, we write $G[V']$ for the subgraph of $G$ induced by $V'$.

A graph $H$ has the \textit{Ramsey property in the hypercube} if, for every
integer $k\geq2$, there is an integer $n_0$ such that every colouring
of $E(Q_n)$ in $k$ colours contains a monochromatic copy of $H$
whenever $n\geq n_0$, see Alon, Radoi\v{c}i\'c, Sudakov and Vondr\'ak~\cite{ARSV}.

For a graph $G=(V,E)$, its \textit{$1$-subdivision} $T_1(G)$ has vertex set $V\cup E$, where $V$ and $E$ are regarded as disjoint sets, and edge set $\{\{u,e\},\{e,v\} : e=uv\in E\}$.
The vertices corresponding to $V$ are called \textit{poles}, and those corresponding to $E$ are called \textit{subdivision vertices}.
For a $d$-regular graph $G$ on $N\geq2$ vertices, let
$\lambda_1\geq\ldots\geq\lambda_N$ be the eigenvalues of its adjacency matrix
and let $\lambda(G)=\max\{|\lambda_2|,|\lambda_N|\}$.
We also write $\mathbb F_q$ for the finite field with $q$ elements, where $q$ is a prime power.

\section{Preliminary results}

We use the following form of the Expander Mixing Lemma, due to
Alon and Chung~\cite[Equation~(2.2)]{AC}.

\begin{lemma}[Alon, Chung \cite{AC}]\label{lemma: expander_mixing}
Let $G$ be a $d$-regular graph on $N$ vertices, where $d\geq1$.
For every $U\subseteq V(G)$,
$   \left|
        e_G(U,V(G)\setminus U)-{d|U|(N-|U|)}/{N}
    \right|
    \leq {\lambda(G)|U|(N-|U|)}/{N}.$
\end{lemma}

\begin{lemma}\label{expander}
For any fixed integers $d,g\geq3$ and all sufficiently large $N$
with $dN$ even, there is a connected $d$-regular graph $G$ on
$N$ vertices of girth at least $g$ such that $\lambda(G)\leq2\sqrt{d-1}+1$ and $\chi(G)\geq\lceil1+d/(2\sqrt{d-1}+1)\rceil$.
\end{lemma}

\begin{proof}[Proof of Lemma~\ref{expander}]
Since $2\sqrt{d-1}<d$, choose
$0<\eps<\min\{1,d-2\sqrt{d-1}\}$.
A uniformly random simple $d$-regular graph on $N$ vertices
has girth at least $g$ with probability tending to a positive
constant by McKay, Wormald and Wysocka~\cite[Corollary~1]{MWW},
and satisfies $\lambda(G)\leq2\sqrt{d-1}+\eps$ with probability
tending to one by Friedman's theorem~\cite{F}, see also
Bordenave~\cite[Theorem~1]{B}.
Thus, for all sufficiently large $N$ with $dN$ even, there is
a graph $G$ satisfying both conditions.
Finally, $G$ is connected. Otherwise, let $U$ be the vertex set
of a connected component of $G$. Then $0<|U|<N$, so Lemma~\ref{lemma: expander_mixing} and $\lambda(G)<d$ give 
$0 < (d-\lambda(G))|U|(1-|U|/N) \leq e_G(U,V(G) \setminus U) = 0$, a contradiction.

Now let $U$ be a largest colour class in a proper $\chi(G)$-colouring
of $G$, so $|U|\geq N/\chi(G)$.
Since $U$ is an independent set and $G$ is $d$-regular, we have that
$e_G(U,V(G) \setminus U)=d|U|$.
By Lemma~\ref{lemma: expander_mixing},
$d|U|=e_G(U,V(G) \setminus U)
\leq(d+\lambda(G))|U|(1-|U|/N)$.
Dividing by $|U|>0$ and using $|U|\geq N/\chi(G)$ gives
$d\leq(d+\lambda(G))(1-1/\chi(G))$.
Finally, rearranging yields
$\chi(G)\geq1+d/\lambda(G)\geq1+d/(2\sqrt{d-1}+1)$ and since $\chi(G)$ is an integer, the claimed bound follows.
\end{proof}

\begin{lemma}\label{lemma: common_core}
Let $G$ be a $d$-regular graph, where
$d\geq24$ and $\lambda(G)\leq d/4$.
In every embedding of $T_1(G)$ into $Q_n$, all poles have a common
neighbour in $Q_n$.
\end{lemma}

\begin{proof}[Proof of Lemma~\ref{lemma: common_core}]
Since $\lambda(G)\leq d/4<d$, the graph $G$ is connected. Let $T$ be a copy of  $T_1(G)$ in $Q_n$ and for each $v\in V(G)$, let  $A_v\in V(T)$ be the pole corresponding to $v$. 
Choose $S\subseteq[n]$ minimising $\sum_{v\in V(G)}|A_v\triangle S|$.
Replacing every vertex $X$ of the embedding by $X\triangle S$,
we may assume that $S=\emptyset$. We shall prove that every $A_v$ is a singleton.
Let  $U_i=\{v:i\in A_v\}$.

\begin{claim}\label{claim: 1}
For each $i\in [n]$, $|U_i|\leq |V(G)|/2$ and $e_G(U_i,V(G) \setminus U_i)\leq |V(G)|$.
\end{claim} 

\begin{proof}[Proof of Claim~\ref{claim: 1}]
Note first that $\sum_{v\in V(G)}|A_v\triangle S|= \sum_{v\in U_i}|A_v\triangle S|
+\sum_{v\in V(G)\setminus U_i}|A_v\triangle S|$. 
Since $S=\emptyset$, replacing $S$ with $S\triangle\{i\}$
decreases the first sum by $|U_i|$ and increases the second
sum by $|V(G)|-|U_i|$,
i.e., it changes $\sum_{v\in V(G)}|A_v\triangle S|$ by $|V(G)|-2|U_i|$.
Thus, by minimality, $|V(G)|-2|U_i|\geq0$, so
$|U_i|\leq |V(G)|/2$.

For each  edge $uv$ of $G$, where $u\in U_i$ and $v\in V(G)\setminus U_i$, the $A_u,A_v$-path, that is a subdivision of $uv$, contains an  edge with direction $i$. 
The edges with direction $i$ in $Q_n$ form a matching, so there are at most $|V(G)|$ of them in $E(T)$. Thus $e_G(U_i,V(G)\setminus U_i)\leq |V(G)|$. 
This proves Claim~\ref{claim: 1}.
\end{proof}

\begin{claim}\label{claim: 2}
$\sum_{v\in V(G)}|A_v|\leq (3/2)\cdot |V(G)|$.
\end{claim} 

\begin{proof}[Proof of Claim~\ref{claim: 2}]
By Lemma~\ref{lemma: expander_mixing},
$ e_G(U_i,V(G)\setminus U_i)\geq(d-\lambda(G))|U_i|(1-|U_i|/|V(G)|)
 \geq (d-\lambda(G))|U_i|/2$, so using Claim~\ref{claim: 1}, 
$|U_i|\leq (2|V(G)|)/(d-\lambda(G)).$
Substituting this last bound into the first inequality and  using $d\geq24$ gives
$e_G(U_i,V(G)\setminus U_i)\geq(d-\lambda(G)-2)|U_i|
 \geq\left({3d}/{4}-2\right)|U_i|
 \geq {2d}|U_i|/3$.
Since for any $uv\in E(G)$, $A_v\triangle A_u=\{i,j\}$  for some $i,j$, we have that $uv$ is between $U_i$ and $V(G)\setminus U_i$ as well as between $U_j$ and $V(G)\setminus U_j$. Thus  $\sum_i e_G(U_i,V(G)\setminus U_i)=2\|G\|=d|V(G)|$. Therefore
$ \sum_{v\in V(G)}|A_v|=\sum_{i\in [n]} |U_i|
 \leq {3}/(2d)\cdot\sum_i e_G(U_i,V(G)\setminus U_i)=3 |V(G)|/2.$ 
 This proves Claim~\ref{claim: 2}.
\end{proof}

\begin{claim}\label{claim: 3}
$|A_v|$ is odd for all $v\in V(G)$.
\end{claim} 

\begin{proof}[Proof of Claim~\ref{claim: 3}]
The graph $G$ is connected and for any $uv\in E(G)$, $|A_v\triangle A_u|=2$. Hence all $|A_v|$'s have the same parity.
If that parity were even, having at most one vertex $v$ with $|A_v|=0$,  we would have 
$
 \sum_{v\in V(G)}|A_v|\geq2(|V(G)|-1)>\frac32|V(G)|,
$
since $|V(G)|\geq d+1\geq25$. This is a contradiction to Claim~\ref{claim: 2}, thus proving Claim~\ref{claim: 3}.
\end{proof}

\begin{claim}\label{claim: 4}
$|A_v|\leq 1$ for all $v\in V(G)$.
\end{claim} 

\begin{proof}[Proof of Claim~\ref{claim: 4}]
Let $R=\{v:|A_v|\geq3\}$. Since for any $v\in V(G)\setminus R$, $|A_v|\leq 1$, by Claim~\ref{claim: 3}, we have $\sum_{v\in V(G)} |A_v| \geq  3\cdot|R| + 1\cdot(|V(G)|-|R|)$.  Together with  Claim~\ref{claim: 2}, we have that $|R|\leq |V(G)|/4$.
Now Lemma~\ref{lemma: expander_mixing} gives
 $e_G(R,V(G)\setminus R)
 \geq{3d}|R|\left(1-|R|/|V(G)|\right)/4
 \geq {9d}|R|/16.$
On the other hand, a set of odd size at least three is at distance
two from a singleton only when it is a $3$-set containing that
singleton. Thus each vertex of $R$ has at most three neighbours
outside $R$ in $G$, giving 
$ e_G(R,V(G)\setminus R)\leq 3|R|$.
Since $9d/16>3$, we obtain $R=\varnothing$. 
This proves Claim~\ref{claim: 4}.
\end{proof}

By Claims~\ref{claim: 3} and~\ref{claim: 4}, every $A_v$ is a singleton. 
Thus $\emptyset$ is a common neighbour of all poles in the translated embedding. 
Undoing the translation proves Lemma~\ref{lemma: common_core}.
\end{proof}

\begin{lemma}\label{lemma: monochromatic_common_core}
Let $G$ be a $d$-regular graph, where $d\geq3$.
Colour $E(Q_n)$ by
$c(\{X,X\cup\{i\}\})=|X\cap[i-1]|\bmod2$.
If a monochromatic copy of $T_1(G)$ has a common neighbour of all
its poles, then $G$ is bipartite.
\end{lemma}

\begin{proof}[Proof of Lemma~\ref{lemma: monochromatic_common_core}]
Let $M\in\{0,1\}$ be the colour of the monochromatic copy,
and let $S\subseteq[n]$ be a common neighbour of its poles.
For every $v\in V(G)$, let $A_v$ denote the corresponding pole,
and for every $uv\in E(G)$, let $A_{uv}$ denote the corresponding
subdivision vertex.
Since every pole is adjacent to $S$, we can write
$A_v=S\triangle\{i_v\}$ for some $i_v\in[n]$.
By injectivity, the directions $i_v$ are pairwise distinct. 
Note that the endpoints of an edge of direction $i$
have the same intersection with $[i-1]$.
Thus its colour can be computed using either endpoint.
Let
$A=\{v\in V(G): |S\cap[i_v-1]|\equiv M\pmod2\}$
and $B=V(G) \setminus A$.
In other words, $A$ consists of those vertices whose
corresponding poles are joined to $S$ by an edge of colour $M$. 

Consider an edge $uv\in E(G)$, and assume without loss of
generality that $i_u<i_v$.
The two common neighbours of $A_u$ and $A_v$ are $S$ and
$S\triangle\{i_u,i_v\}$, so $A_{uv}$ is equal to one of
these vertices.

\caseheading $A_{uv}=S$.
The two incident edges have colours
$|S\cap[i_u-1]|\bmod 2$ and $|S\cap[i_v-1]|\bmod 2$, respectively.
Since both colours are $M$, we have $u,v\in A$.
By injectivity, at most one subdivision vertex has image $S$, so this case occurs for at most one edge of $G$.

\caseheading $A_{uv}=S\triangle\{i_u,i_v\}$.
Note that the edge $\{A_u,A_{uv}\}$ has direction $i_v$.
Since $A_u=S\triangle\{i_u\}$ and $i_u<i_v$,
the numbers $|A_u\cap[i_v-1]|$ and $|S\cap[i_v-1]|$
have opposite parity.
Hence this edge has colour $(|S\cap[i_v-1]|+1)\bmod 2$.
Note next that the edge $\{A_v,A_{uv}\}$ has direction $i_u$.
Since $A_v=S\triangle\{i_v\}$ and $i_v>i_u$,
we have $A_v\cap[i_u-1]=S\cap[i_u-1]$.
So this edge has colour $|S\cap[i_u-1]|\bmod 2$.
Since both colours are $M$, we have $u\in A$ and $v\in B$.

Consequently, every edge of $G$ has at most one endpoint in $B$, so $B$ is
independent in $G$, and the edges of $G[A]$ are exactly the edges of the
first case, so $\|G[A]\|\leq1$.

We shall now show that $A$ is also independent in $G$.
Since $G$ is $d$-regular, summing the degrees in the two parts gives $d|A|=e_G(A,B)+2\|G[A]\|$ and $d|B|=e_G(A,B)$. 
Here, each edge between $A$ and $B$ is counted once in each
degree sum, while each edge of $G[A]$ is counted twice in
the first sum.
The second equality also uses that $B$ is independent.
Subtracting the second equality from the first yields $d(|A|-|B|)=2\|G[A]\|$.
Since $\|G[A]\|\leq1$, the right-hand side is either zero or two.
If $G[A]$ has exactly one edge, then $d(|A|-|B|)=2$.
This is impossible, since $|A|-|B|$ is an integer and $d\geq3$.
Thus $A$ is also an independent set, so $G$ is bipartite.
This completes the proof of Lemma~\ref{lemma: monochromatic_common_core}.
\end{proof}

The proof of the following proposition is strongly inspired by the construction of Ellis, Ivan and Leader~\cite{EIL} for daisy-free hypergraphs.
Alon~\cite{A-erasure} and, independently, Johnson,
see the note at the end of~\cite{EIL}, observed that this
construction also gives a positive lower bound, independent
of the uniformity, for the hypergraph Tur\'an densities of
augmentations of graphs with chromatic number at least four.
We combine this observation with the assumption that all poles have
a common neighbour in every embedding into a hypercube to
obtain the following lower bound.

\begin{proposition}\label{proposition: chromatic_subdivisions_lower}
Let $G$ be a graph such that the poles have a common neighbour in every
embedding of $T_1(G)$ into a hypercube. Let $q\geq2$ be a prime power with
$\chi(G)\geq2q+4$. For every $n\geq1$, there is a set of vertices
$\mathcal X\subseteq V(Q_n)$ such that $Q_n[\mathcal X]$ is
$T_1(G)$-free and $\|Q_n[\mathcal X]\|\geq\beta(q)^4\|Q_n\|$,
where $\beta(q)=\prod_{j=1}^{\infty}(1-q^{-j}) > 0$.
\end{proposition}

\begin{proof}[Proof of Proposition~\ref{proposition: chromatic_subdivisions_lower}]
Fix $n\geq1$ and fix $k$, $0\leq k\leq n$. For each  $i\in[n]$, choose vectors 
$a_{k,i},b_{k,i}\in\mathbb F_q^k$, all independently and uniformly
at random. I.e., the probability that $a_{k,i}=\mathbf{v}$, for some specific ${\mathbf v}\in \mathbb F_q^k$ is $1/q^k$.
Repeat this procedure for each $k$, $0\leq k\leq n$.

Let $\mathcal X$ consist of all vertices $X\subseteq[n]$
for which the multiset  $\{a_{|X|,i}: {i\in X}\}$ is a basis of
$\mathbb F_q^{|X|}$ and the multiset $\{b_{n-|X|,i}: i\in[n] \setminus X\}$ is a basis
of $\mathbb F_q^{n-|X|}$.
The empty family is the basis of $\mathbb F_q^0$.

Let $p_k=\prod_{j=1}^k(1-q^{-j})$, with $p_0=1$.
A multiset  of $k$ vectors chosen uniformly at random from $\mathbb F_q^k$ as described above  
 is a basis with probability $p_k\geq\beta(q)$. 
Indeed, conditional on the first $r$ vectors being linearly independent,
the next one avoids their span with probability $1-q^{r-k}$.
In particular, $\Pr(X\in\mathcal X)=p_{|X|}p_{n-|X|}
\geq\beta(q)^2$ for every $X\subseteq[n]$.

\begin{claim}\label{claim: t-free}
For every choice of the vectors, $Q_n[\mathcal X]$ is $T_1(G)$-free.
\end{claim}

\begin{proof}[Proof of Claim~\ref{claim: t-free}]
Suppose that $Q_n[\mathcal X]$ contains a copy $T$ of $T_1(G)$
with poles $A_v$ and subdivision vertices $A_{uv}$.
By assumption, $A_v=S\triangle\{i_v\}$ for a common neighbour
$S$ in $Q_n$, where the directions $i_v$ are distinct by injectivity.
The two common neighbours of $A_u$ and $A_v$ are $S$ and
$S\triangle\{i_u,i_v\}$, so $A_{uv}$ equals one of these vertices.
By injectivity, at most one subdivision vertex equals $S$.
If there is such a vertex, corresponding to the edge $e$ of $G$,  
let $G'=G-e$; otherwise let $G'=G$.
Then $\chi(G)\leq\chi(G')+1$ and
$A_{uv}=S\triangle\{i_u,i_v\}$ for every $uv\in E(G')$.

Partition $V(G')$ into
$V^+=\{v\in V(G') : i_v\notin S\}$ and
$V^-=\{v\in V(G') : i_v\in S\}$.
We first show that $\chi(G'[V^+])\leq q+1$.
If this graph has no edges, this is immediate.
Otherwise, let $k=|S|+2\leq n$.
For every $uv\in E(G'[V^+])$, the set $S\cup\{i_u,i_v\}$ is a vertex of $T$, i.e., it  belongs to $\mathcal X$, so $(a_{k,i})_{i\in S\cup\{i_u,i_v\}}$
is a basis of $\mathbb F_q^k$.
Fix one such basis and order its vectors such that the vectors $a_{k,i}$ for $i\in S$ come first. 
Let $\rho:\mathbb F_q^k\to\mathbb F_q^2$ map each vector to the pair of its last two coordinates with respect to this basis.
Then $\ker \rho=W$, where $W=\operatorname{span}\{a_{k,i}:i\in S\}$.

For every $v\in V^+$, write
$z_v=\rho(a_{k,i_v})=(x_v,y_v)\in\mathbb F_q^2$.
If $uv\in E(G'[V^+])$ and $\lambda_1z_u+\lambda_2z_v=0$, then
$\lambda_1a_{k,i_u}+\lambda_2a_{k,i_v}\in W$.
Since $(a_{k,i})_{i\in S\cup\{i_u,i_v\}}$ is a basis, this
implies $\lambda_1=\lambda_2=0$.
Thus $z_u$ and $z_v$ are linearly independent.
Colour each vertex $v$ with $z_v\neq0$ by $y_v/x_v$ if
$x_v\neq0$, and by $\infty$ otherwise.
This uses the $q+1$ colours in $\mathbb F_q\cup\{\infty\}$,
and two nonzero vectors give the same colour exactly when they
are scalar multiples.
Hence adjacent vertices receive different colours.
Vertices with zero image are isolated in this part and receive
any colour, giving $\chi(G'[V^+])\leq q+1$.

For $G'[V^-]$, the bound is again immediate if it has no edges.
Otherwise, let $S'=[n] \setminus S$ and $k'=|S'|+2\leq n$.
For every $uv\in E(G'[V^-])$, the subdivision vertex
$A_{uv}=S \setminus \{i_u,i_v\}$ has complement $S'\cup\{i_u,i_v\}$.
Since $A_{uv}\in\mathcal X$, the vectors
$(b_{k',i})_{i\in S'\cup\{i_u,i_v\}}$ form a basis of
$\mathbb F_q^{k'}$.
Now the same argument with $S'$ instead of $S$ and $b_{k',i}$
instead of $a_{k,i}$ gives $\chi(G'[V^-])\leq q+1$.
Colouring the two parts with disjoint sets of colours gives
$\chi(G') \leq 2q+2$. 
Hence $\chi(G)\leq\chi(G')+1\leq2q+3$, a contradiction. This proves Claim~\ref{claim: t-free}.
\end{proof}

\begin{claim}\label{claim: edges_in_x}
There is a choice of the vectors for which
$\|Q_n[\mathcal X]\|\geq\beta(q)^4\|Q_n\|$.
\end{claim}

\begin{proof}[Proof of Claim~\ref{claim: edges_in_x}]
Recall that $\Pr(X\in\mathcal X) \geq \beta(q)^2$. 
Fix an edge $XY\in E(Q_n)$ with $|X|=k$ and $Y=X\cup\{i\}$.
The event $X\in\mathcal X$ depends only on
$(a_{k,j})_{j\in X}$ and $(b_{n-k,j})_{j\in[n] \setminus X}$,
whereas $Y\in\mathcal X$ depends only on
$(a_{k+1,j})_{j\in Y}$ and $(b_{n-k-1,j})_{j\in[n] \setminus Y}$.
These four families are chosen independently even if some of the
dimensions coincide. For example, if $n-k=k+1$, the vectors
$b_{n-k,j}$ and $a_{k+1,j}$ still belong to independently
chosen families. Thus the events $X\in\mathcal X$ and
$Y\in\mathcal X$ are independent.
It follows that $\Pr(X\in\mathcal X\text{ and }Y\in\mathcal X)
= \Pr(X\in\mathcal X)\Pr(Y\in\mathcal X)
\geq\beta(q)^4$.
Summing over all edges and using linearity of expectation gives
$\mathbb E[\|Q_n[\mathcal X]\|]\geq\beta(q)^4\|Q_n\|$.
Hence some choice of the vectors gives the required bound. This proves Claim~\ref{claim: edges_in_x}.
\end{proof}

Together, Claims~\ref{claim: t-free} and~\ref{claim: edges_in_x}
prove Proposition~\ref{proposition: chromatic_subdivisions_lower}.
\end{proof}

\section{Proof of Theorem~\ref{main}}

Fix integers $d\geq419$ and $g\geq3$.
By Lemma~\ref{expander}, there is a connected $d$-regular graph
$H=H(d,g)$ of girth at least $g$ satisfying
$\lambda(H)\leq2\sqrt{d-1}+1<d/4$ and
$\chi(H)\geq\lceil1+d/(2\sqrt{d-1}+1)\rceil\geq12$.
Let $q$ be the largest prime power at most $(\chi(H)-4)/2$.
Then $q\geq4$ and $\chi(H)\geq2q+4$.
The graph $T_1(H)$ is connected and has girth at least $2g$.
It is also layered, since mapping each pole $v$ to $\{v\}$
and each subdivision vertex $uv$ to $\{u,v\}$ gives an induced
embedding into the second edge layer of a hypercube.
Lemma~\ref{lemma: common_core} and
Proposition~\ref{proposition: chromatic_subdivisions_lower} give
$\ex(Q_n,T_1(H))\geq\beta(q)^4\|Q_n\|$ for every $n\geq1$.
This implies $\pi_\square(T_1(H))\geq\beta(q)^4$.

To obtain a lower bound in terms of $d$, note that, for every
integer $m\geq1$, induction gives $\prod_{j=1}^m(1-q^{-j})\geq1-\sum_{j=1}^m q^{-j}$. 
Letting $m\rightarrow\infty$, we obtain
$\beta(q)\geq1-\sum_{j\geq1}q^{-j}=1-1/(q-1)$.
Since $q\geq4$, the right-hand side is positive, so
Bernoulli's inequality gives
$\beta(q)^4\geq(1-1/(q-1))^4\geq1-4/(q-1)$.
To relate $q$ to $\chi(H)$, note that there is a power of two
greater than $(\chi(H)-4)/4$ and at most $(\chi(H)-4)/2$.
Since this is a prime power, the choice of $q$ gives
$q>(\chi(H)-4)/4$, and hence
$\beta(q)^4>1-16/(\chi(H)-8)$.
Together with the lower bound on $\chi(H)$, this gives
$\beta(q)^4>1-1/f(d)$, where
$
    f(d)=\frac{1}{16}\left(\frac{d}{2\sqrt{d-1}+1}-7\right).
$
This function is positive for $d\geq419$ and satisfies
$f(d)\rightarrow\infty$ as $d\rightarrow\infty$.

To complement the bound $1-1/f(d)$, which is negative for
small $d$, we next show that $T_1(H)$ does not have the Ramsey
property and deduce the uniform lower bound
$\pi_\square(T_1(H))\geq1/2$.
For every $n\geq1$, colour $E(Q_n)$ by
$c(\{X,X\cup\{i\}\})=|X\cap[i-1]|\bmod2$.
Suppose that this colouring contains a monochromatic copy
of $T_1(H)$.
By Lemma~\ref{lemma: common_core}, all poles of this copy
have a common neighbour.
Lemma~\ref{lemma: monochromatic_common_core} then implies
that $H$ is bipartite, a contradiction to $\chi(H)\geq2q+4$.
Thus this colouring contains no monochromatic copy of $T_1(H)$,
so $T_1(H)$ does not have the Ramsey property.
A largest colour class gives a $T_1(H)$-free subgraph
with at least $\|Q_n\|/2$ edges, and hence
$\pi_\square(T_1(H))\geq1/2$.
Combining both bounds gives $\pi_\square(T_1(H))\geq\max\{1/2,\beta(q)^4\}$. 
This completes the proof of Theorem~\ref{main}. \qed

\section{Concluding remarks}

We proved that layered graphs of arbitrarily high girth can have Tur\'an densities arbitrarily close to one, even when an edge colouring of every hypercube in two colours avoids monochromatic copies of these graphs.
We did not attempt to optimise the parameters of $H(d,g)$, such as the degree $d$.

In addition to the extremal hypercube Tur\'an function $\ex(Q_n, H)$,  for a graph $H$, that measures the largest number of edges in an $H$-free subgraph of $Q_n$,  the vertex version $\ex_v(Q_n, H)$  is also natural to consider and is defined as the largest number of vertices in $Q_n$ whose induced subgraph contains no copy of $H$.  Letting the respective density be
$$\pi_{v, \square}(H) = \lim _{n \rightarrow \infty} \ex_v(Q_n, H)/|V(Q_n)|,$$ 
our proofs immediately give  the following vertex-analogue of Theorem \ref{main}.

\begin{proposition}
There is a positive function $f^*$ with $f^*(d)\rightarrow\infty$
as $d\rightarrow\infty$ such that, for any integers
$d\geq419$ and $g\geq3$,
$ \pi_{v,\square}(T_1(H(d,g)))
    \geq \max\{1/2,\beta(q)^2\}
    > 1-1/f^*(d).$
\end{proposition}

We note that here the lower bound $1/2$ does not come from the Ramsey
property but is trivial for graphs with at least one edge, since the vertices corresponding to sets of even size form an independent set of size $2^{n-1}$, so $\pi_{v,\square}(H)\geq1/2$.

Finally, we also note that, in general, if a fixed cubical graph $H$ does not have the Ramsey property, then there exists a fixed $s$ such that, for every $n$, there is an $s$-colouring of $E(Q_n)$ containing no monochromatic copy of $H$.
Taking a largest colour class gives $\pi_\square(H)\geq1/s$.
Alon, Radoi\v{c}i\'c, Sudakov and Vondr\'ak~\cite{ARSV} characterised the graphs having the Ramsey property.
A graph has this property if and only if it has an embedding in an edge layer of a hypercube in which all edges of the same direction have the same \textit{prefix sum}, where the prefix sum of $\{X,X\cup\{i\}\}$ is $|X\cap[i-1]|$.

For a connected cubical graph $H$ without the Ramsey property, their proof gives a colouring avoiding monochromatic copies of $H$ with at most $2\lceil\operatorname{diam}(H)/2\rceil$ colours.
For fixed $d$ and $g$, the base graph $H(d,g)$ can be chosen
with arbitrarily many vertices.
Since $T_1(H(d,g))$ has maximum degree $d$, its diameter is
$\Omega_d(\log|V(T_1(H(d,g)))|)$.
As the number of vertices grows, the resulting bound for the Tur\'an density in terms of diameter tends to zero, whereas our bound
$\max\{1/2,\beta(q)^4\}$ remains positive.

{\bf Acknowledgements.} Some parts of the proof were obtained with assistance from GPT-6 Astra (OpenAI). All the arguments were verified and rewritten by the authors.

}

\end{document}